\documentclass[10pt]{article}
\usepackage{tocloft}

\usepackage{latexsym,amsfonts,euscript,amsthm}

\usepackage{amssymb}
\usepackage{stmaryrd}
\usepackage{mathrsfs,amsmath}
\usepackage{leftidx}
\usepackage{mathabx,enumerate}
\usepackage[shortlabels]{enumitem}

\usepackage[figuresright]{rotating}
\usepackage[table]{xcolor}
\usepackage{multirow}
\usepackage{booktabs}
\usepackage{tabularx}

\usepackage{tocbibind} 

\usepackage{float}
\usepackage{tikz-cd}
\usetikzlibrary{cd}

\usepackage[pagebackref]{hyperref}

\usepackage{hyperref}

\hypersetup{%
  colorlinks=false,
  linktoc=all,
  citebordercolor=green,
  linkbordercolor=red,
  urlbordercolor=blue, 
  pdfborder={0 0 1.5}, 
  citecolor=black,
    linkcolor=black,
  urlcolor=black
}

\makeatletter
\renewcommand{\@maketitle}{%
  \newpage
  \null
  \vskip 2em%
  \begin{center}%
    \let \footnote \thanks
    {\huge \bfseries \@title \par}
    \vskip 1.5em%
    {\large \textsc{\@author}\par}%
    \vskip 1em%
    {\large \@date \par}%
  \end{center}%
  \par
  \vskip 1.5em
}
\makeatother

\renewcommand{\contentsname}{\small\bfseries Contents}

\newcommand{\centeredtitle}[1]{%
  \begin{center}
    {\small\bfseries #1}%
  \end{center}
  \vspace{1ex}%
}

\makeatletter
\AtBeginDocument{%
  \renewcommand{\tableofcontents}{%
    \centeredtitle{\contentsname}%
    \@starttoc{toc}%
  }%
}
\makeatother

\renewenvironment{abstract}{%
  \centeredtitle{\abstractname}%
  \vspace{-0.6\baselineskip}
  \quotation
}{%
  \endquotation
}

\newcommand{\ackname}{Acknowledgements}
\makeatletter
\if@titlepage
  \newenvironment{acknowledgement}{%
    \titlepage
    \null\vfil
    \@beginparpenalty\@lowpenalty
    \begin{center}%
      \bfseries \ackname
      \@endparpenalty\@M
    \end{center}}%
  {\par\vfil\null\endtitlepage}
\else
  \newenvironment{acknowledgement}{%
    \if@twocolumn
      \section*{\ackname}%
    \else
      \small
      \begin{center}%
        {\bfseries \ackname\vspace{-0.5em}\vspace{\z@}}%
      \end{center}%
      \quotation
    \fi}
    {\if@twocolumn\else\endquotation\fi}
\fi
\makeatother

\makeatletter 
\@addtoreset{equation}{section}
\makeatother  

\makeatletter
\def\thanks#1{\g@addto@macro\@thanks{\footnotetext{#1}}}
\makeatother

\makeatletter
\newtoks\address@list
\newcommand{\addaddress}[3]{%
  \address@list=\expandafter{\the\address@list
    \textsc{#1, #2.} \par
    \textit{Email address}: \href{mailto:#3}{\textsf{#3}} \par
    \addvspace{\medskipamount}%
  }%
}
\newcommand{\printaddresses}{%
  \AtEndDocument{\bigskip{\footnotesize
    \the\address@list
  }}%
}
\makeatother

\newtheorem{lem}{\bf Lemma}[section]

\newtheorem{thm}[lem]{\bf Theorem}

\newtheorem{mainthm}{Theorem}

\newtheorem{maincor}[mainthm]{Corollary}

\title{Finite groups with a unique real $2$-block
\thanks{
\textbf{2020 MR Subject Classification}\,\, Primary 20C20, 20C15.
}}

\author{\textsc{Yu Zeng}, \textsc{Fuming Jiang}*\thanks{*Corresponding author.}\\
}
\addaddress{Yu Zeng}{Department of Mathematics, Suzhou University of Technology, Changshu, Jiangsu, 215500, China}{yuzeng2004@163.com}
\addaddress{Fuming Jiang}{School of Mathematics and Statistics, Southwest University, Chongqing, 400715, China}{fuming.jiang@hotmail.com}

\printaddresses

\date{}

\begin{document}

\maketitle

\begin{abstract}
  Let $p$ be a prime, and let $B$ be a $p$-block of a finite group $G$.
  A $p$-block $B$ is called \emph{real} if the set of irreducible ordinary characters contained in $B$ 
  is invariant under complex conjugation.
  Motivated by Harris' classification of the finite groups with a unique $p$-block for an arbitrary prime $p$,
  and by McHugh and Schaeffer Fry's classification of the finite quasi-simple groups with a unique real $2$-block,
  we classify, in this paper, all finite groups admitting exactly one real $2$-block.
\end{abstract}


\section{Introduction}

Given a prime $p$, classifying finite groups by the number of their $p$-blocks is a classical and fundamental problem in the representation theory of finite groups, with particular interest in groups admitting few $p$-blocks.
A landmark result in this direction is due to Harris \cite{harris85}, who completely classified finite groups with exactly one $p$-block for an arbitrary prime $p$.

Let $B$ be a $p$-block of a finite group $G$ for a prime $p$, and let $\mathrm{Irr}(B)$ denote the set of irreducible ordinary characters belonging to $B$.
 A block $B$ is called \emph{real} if the complex conjugate $\overline{\chi}$ lies in $\mathrm{Irr}(B)$ for every $\chi\in\mathrm{Irr}(B)$. 
 The notion of real $p$-blocks was first introduced by Brauer \cite[Page 518]{brauer71}, who established a general framework and a series of fundamental theorems for real $p$-blocks for arbitrary primes $p$. 
 Among all primes, the theory of real $2$-blocks is the most refined and well-developed. 
 Specifically, a $2$-block is real if and only if it contains at least one real-valued irreducible ordinary character, a characterization unique to the prime $2$.

Harris's classification of finite groups with a unique $p$-block naturally suggests a refined problem: classifying finite groups possessing exactly one real $2$-block. 
Recently, McHugh and Schaeffer Fry \cite{mchugh26} classified all finite quasi-simple groups satisfying this condition.
In the present paper, we continue this line of investigation and classify all finite groups admitting a unique real $2$-block.

Before stating our main results, we recall several standard definitions.
Let $G$ be a finite group and $p$ a prime.
We say $G$ is \emph{$p$-constrained} if $\mathrm{C}_{G}(\mathrm{O}_{p',p}(G)/\mathrm{O}_{p'}(G))\leq \mathrm{O}_{p',p}(G)$.
A finite group is \emph{quasi-simple} if it is a perfect central extension of a finite non-abelian simple group.
A \emph{component} $L$ of $G$ is a subnormal quasi-simple subgroup,
and if $L/\mathrm{Z}(L)\cong S$ for a finite non-abelian simple group $S$, 
then we say $L$ is of \emph{type} $S$.

\begin{mainthm}\label{thmA}
	Let $G$ be a finite group, and let $N=\mathrm{O}^{2'}(G)$.
  Then $G$ has a unique real $2$-block if and only if 
  the following hold.
  \begin{enumerate}[\rm (1)]
    			\item  Every $N$-conjugacy class contained in $\mathrm{O}_{2'}(N)$ has odd size. 
     \item Set $T=N/\mathrm{O}_{2'}(N)$. 
     Then $T=E\times H$ where $E$ and $H$ satisfy the following condition.
     \begin{enumerate}[\rm (i)]
      \item Either $E=1$, or $E=S_1\times S_2\times \cdots \times S_t$ where $S_i\in \{ M_{11}, M_{23}, \mathrm{PSL}_3(3), \mathrm{PSU}_3(3)\}$ for $1\leq i\leq t$.
\item Either $H$ is $2$-constrained, or all components of $H$ are of type $M_{22}$ or $M_{24}$ (that is, $L/\mathrm{Z}(L)\in \{ M_{22}, M_{24} \}$ for all components $L$ of $H$).
     \end{enumerate}
  \end{enumerate}
\end{mainthm}

By a classical result of Gow \cite[Theorem 5.1]{gow79}, every real $2$-block contains 
a real-valued irreducible 2-rational ordinary character with Schur index $1$ over $\mathbb{Q}$.
 This yields the following immediate consequence of Theorem \ref{thmA}.

\begin{maincor}\label{corB}
	Let $G$ be a finite group, and let $N=\mathrm{O}^{2'}(G)$.
  Then every 
  real-valued irreducible $2$-rational ordinary character of $G$ with Schur index $1$ over $\mathbb{Q}$
  lies in the principal $2$-block of $G$
   if and only if 
  the following hold.
  \begin{enumerate}[\rm (1)]
    			\item  Every $N$-conjugacy class contained in $\mathrm{O}_{2'}(N)$ has odd size. 
     \item Set $T=N/\mathrm{O}_{2'}(N)$. 
     Then $T=E\times H$ where $E$ and $H$ satisfy the following condition.
     \begin{enumerate}[\rm (i)]
      \item Either $E=1$, or $E=S_1\times S_2\times \cdots \times S_t$ where $S_i\in \{ M_{11}, M_{23}, \mathrm{PSL}_3(3), \mathrm{PSU}_3(3)\}$ for $1\leq i\leq t$.
\item Either $H$ is $2$-constrained, or all components of $H$ are of type $M_{22}$ or $M_{24}$ (that is, $L/\mathrm{Z}(L)\in \{ M_{22}, M_{24} \}$ for all components $L$ of $H$).
     \end{enumerate}
  \end{enumerate}
\end{maincor}

The paper is organized as follows. 
In Section 2, we gather preliminary notions and results concerning characters and blocks of finite groups and of finite groups of Lie type, which will be used throughout the paper.
In Section 3, 
we establish several results on real blocks, including a correspondence theorem for real $2$-blocks 
between a finite group and specific quotient groups of it.
Finally, in Section 4, we prove our main results.

\section{Preliminaries}

We begin by collecting a number of preliminary notions and results that will be useful in what follows.
Throughout the paper, $G$ always stands for a finite group, while $p$ and $\ell$ always denote primes.
Unless stated otherwise, whenever we refer simply to (irreducible) characters, we mean (irreducible) ordinary characters;
when we refer to blocks of a finite group, we mean $p$-blocks.
Given a positive integer $n$ and a prime $p$, write $n_p$ for the largest power of $p$ dividing $n$, and $n_{p'}$ for the $p'$-part $n/n_p$ of $n$.
Furthermore, we refer to \cite{huppert67,kurzweil04} as sources for finite group theory.

\subsection{Characters and conjugacy classes}

We first recall some basic notions and facts from the character theory of finite groups,
and we refer to \cite{isaacs76,huppert98,navarro18} for further details.

Let $\mathrm{Irr}(G)$ denote the set of irreducible (ordinary) characters of a finite group $G$.
For a character $\chi$ of $G$, we write $\mathbb{Q}(\chi)$ for the \emph{field of values} of $\chi$, that is, the subfield of $\mathbb{C}$ generated by the values $\chi(g)$ for all $g \in G$.
We denote by $\overline{\chi}$ the \emph{complex conjugate character} of $\chi$, defined by $\overline{\chi}(g) = \chi(g^{-1})$ for every $g \in G$.
It is well known that $\chi(g^{-1})=\overline{\chi(g)}$ for all $g\in G$, so this definition agrees with taking the pointwise complex conjugate of $\chi$.
A character $\chi$ is said to be \emph{real-valued} if $\mathbb{Q}(\chi) \subseteq \mathbb{R}$, or equivalently, if $\overline{\chi} = \chi$.

Given an element $x \in G$, let $x^G$ denote its $G$-conjugacy class.
We call $x^G$ a \emph{real $G$-conjugacy class} if $x^{-1} \in x^G$; in this case, $x$ is also called a \emph{real element} of $G$.
It is a standard fact that the number of real-valued characters in $\mathrm{Irr}(G)$ equals 
the number of real conjugacy classes of $G$.
As a consequence, an odd-order group has a unique real-valued irreducible character i.e. the principal character.

We fix the following notation for characters relative to normal subgroups, which will be employed consistently throughout the paper.
Let $N \unlhd G$ and let $\theta \in \mathrm{Irr}(N)$.
We always identify each character in $\mathrm{Irr}(G/N)$ with its inflation to $G$, thereby viewing $\mathrm{Irr}(G/N)$ as a subset of $\mathrm{Irr}(G)$. 
We denote by $\mathrm{Irr}(G | \theta)$ the set of irreducible characters of $G$ lying over $\theta$, i.e.
\[
\mathrm{Irr}(G|\theta)=\{ \chi\in \mathrm{Irr}(G): [\chi_N,\theta]>0 \}.
\] 
The inertia subgroup of $\theta$ in $G$ is denoted by $\mathrm{I}_G(\theta)$.

We begin with an elementary observation about real-valued characters under restriction and induction.

\begin{lem}\label{lem: res and ind}
  Let $G$ be a finite group admitting two subgroups $H$ and $K$ such that $H\leq K$.
  If $\chi\in \mathrm{Irr}(G)$ is real-valued, then the character $(\chi_H)^{K}$ is also real-valued.
\end{lem}
\begin{proof}
  This follows directly from the inclusions of fields of values $\mathbb{Q}((\chi_H)^{K})\subseteq \mathbb{Q}(\chi_H)\subseteq \mathbb{Q}(\chi)\subseteq \mathbb{R}$.
\end{proof}

We now turn to a result on tensor induced characters, which will serve as a key technical tool in our later  
arguments for non-abelian minimal normal subgroups.

\begin{lem}\label{lem: tensor induction}
	Let $N=S_1\times \cdots \times S_t$ be a non-abelian minimal normal subgroup
	of a finite group $G$,
  where $S_i=S^{x_i}$ for some $x_i\in G$ and $x_1=1$.
  Suppose $\theta=\alpha_1\times \cdots \times \alpha_t\in \mathrm{Irr}(N)$ is $G$-invariant where $\alpha_i\in \mathrm{Irr}(S_i)$, and set $\alpha=\alpha_1$.
    Assume that $\alpha \times 1_{\mathrm{C}_{G}(S)}$ extends to $\psi \in \mathrm{Irr}(\mathrm{N}_{G}(S))$
	such that $\mathbb{Q}(\psi)=\mathbb{Q}(\alpha)$.
	Then $\theta$ extends to the tensor induced character $\chi:=\psi^{\otimes G}\in \mathrm{Irr}(G)$, such that $\mathbb{Q}(\chi) =\mathbb{Q}(\theta)$ and $\mathrm{C}_{G}(N)\leq \ker(\chi)$.
  If, in addition, the cyclic group $\langle \sigma\rangle$ acts via automorphisms on $G$, then $\chi^\sigma=(\psi^\sigma)^{\otimes G}\in \mathrm{Irr}(G)$.
\end{lem}
\begin{proof}
	Set $ H = \mathrm{N}_G(S) $. 
  Then $ \{x_1, x_2,\ldots, x_t\} $ is a transversal for $ H $ in $ G $.
  Since $\theta$ is $G$-invariant, we have $\alpha_i=\alpha^{x_i}$.
  Write $\chi=\psi^{\otimes G}$.
	As in the proof of \cite[Corollary 10.5]{navarro18}, we have $\chi_N=\theta$.
  In particular, $\chi$ is an irreducible character of $G$.

  Set $\Omega=\{ 1,2,\cdots ,t \}$.
  For $g\in G$,
  let 
   \begin{equation}\label{eq}
      (a_i, a_i \pi(g), \ldots, a_i \pi(g)^{f_i - 1}),\quad i=1,\dots,z(g),
   \end{equation}
  be the cycles of the permutation $\pi(g)\in \mathsf{Sym}(\Omega)$ (the symmetric group on $\Omega$)
   defined by $x_j g = h_j(g)x_{j\pi(g)}$, where $h_j(g)\in H$.
  Then, by \cite[Theorem 25.3]{huppert98},
	\[
	\chi(g)=\prod_{i=1}^{z(g)} \psi(h_{a_i}(g)h_{a_i\pi(g)}(g)\cdots h_{a_i\pi(g)^{f_i-1}}(g)).
	\]
		Consequently,
	$$\mathbb{Q}(\theta)=\mathbb{Q}(\chi_N)\subseteq \mathbb{Q}(\chi)\subseteq \mathbb{Q}(\psi)=\mathbb{Q}(\alpha)\subseteq \mathbb{Q}(\theta),$$
  so $\mathbb{Q}(\chi)=\mathbb{Q}(\theta)$.
    
  Now, let $g\in \mathrm{C}_{G}(N)$.
	Since $\mathrm{C}_{G}(N)\unlhd G$, 
  we have $x_jg=h_j(g)x_j$, hence $h_j(g)=x_jg x_j^{-1}\in \mathrm{C}_{H}(N)$.
  Therefore, 
	\[
	 \chi(g)=\prod_{j=1}^{t} \psi(h_j(g))=\prod_{j=1}^{t} \psi(x_jg x_j^{-1})=\prod_{j=1}^{t} \psi(1)=\chi(1)
	\]
    where the third equality holds as $x_jg x_j^{-1}\in \mathrm{C}_{G}(N)\leq \mathrm{C}_G(S)\leq \ker(\psi)$.
    Thus $\mathrm{C}_{G}(N)\leq \ker(\chi)$.

Finally, suppose that the cyclic group $\langle \sigma\rangle$ acts on $G$.
  We retain all notation from above.
  Then $\psi^\sigma\in \mathrm{Irr}(H^\sigma)$, and $\{ x_1^{\sigma},\cdots , x_t^{\sigma} \}$ is a transversal of $H^{\sigma}$ in $G$.
  Note that $x_j^{\sigma} g^\sigma=h_j(g)^\sigma x_{j\pi(g)}^{\sigma}$ with $h_j(g)\in H$,
  and so $g^{\sigma}$ induces the same permutation on $\Omega$ as $g$.
  By abuse of notation, we write $\pi(g^{\sigma})=\pi(g)\in \mathsf{Sym}(\Omega)$.
  In particular, $\pi(g^{\sigma})$ admits the cycle decomposition (\ref{eq}).
  By \cite[Theorem 25.3]{huppert98}, we have
  \[
      \begin{aligned}
          (\psi^\sigma)^{\otimes G}(g^{\sigma})&=\prod_{i=1}^{z(g)} \psi^{\sigma}(h_{a_i}(g)^{\sigma}h_{a_i\pi(g)}(g)^{\sigma}\cdots h_{a_i\pi(g)^{f_i-1}}(g)^{\sigma})\\
         & =\prod_{i=1}^{z(g)} \psi(h_{a_i}(g)h_{a_i\pi(g)}(g)\cdots h_{a_i\pi(g)^{f_i-1}}(g))=\chi(g)=\chi^\sigma(g^{\sigma}).
          \end{aligned}
  \]
Consequently, $\chi^\sigma=(\psi^\sigma)^{\otimes G}\in \mathrm{Irr}(G)$, as desired.
\end{proof}

We close this subsection with a lemma that collects two useful facts about real-valued characters relative to a normal subgroup.

\begin{lem}\label{lem: real chars}
Let $N$ be a normal subgroup of a finite group $G$ such that $[G:N]$ is odd.  
Then the following statements hold.
\begin{enumerate}[\rm (1)]
    \item 
    If $\chi \in \mathrm{Irr}(G)$ is real-valued, then every irreducible constituent of $\chi_N$ is real-valued.
    In particular, $\chi_N$ is a sum of a $G$-orbit of real-valued irreducible characters of $N$.
    \item 
    If $\theta \in \mathrm{Irr}(N)$ is real-valued, then there exists a unique real-valued character $\chi \in \mathrm{Irr}(G|\theta)$.
\end{enumerate}
\end{lem}
\begin{proof}
  Part (2) is \cite[Corollary 1]{richards85},
  so we only prove part (1) below.

   For part (1), as $[G:N]$ is odd,
   we have $\chi_N=e(\theta_1+\cdots +\theta_t)$ where $\theta_i\in \mathrm{Irr}(N)$ and $t=[G:\mathrm{I}_{G}(\theta_1)]$ is odd.
  Since $\chi$ is real-valued, complex conjugation permutes the set $\{\theta_1, \dots, \theta_t\}$. 
  This permutation has order at most $2$, and since $t$ is odd, it must admit at least one fixed point. 
That is, there exists some $1\leq j\leq t$ such that $\overline{\theta_j} = \theta_j$, i.e. $\theta_j$ is real-valued.
As all $\theta_i$ are $G$-conjugate to $\theta_j$, and the property of being real-valued is preserved under $G$-conjugation, every irreducible constituent of $\chi_N$ is real-valued.
\end{proof}

\subsection{Blocks}

We next recall some basic notions and facts of block theory for finite groups, and refer to \cite{nagao89, navarro98} for further details.

Let $p$ be a prime and let $G$ be a finite group. 
Let $R$ denote the ring of algebraic integers in $\mathbb{C}$, and fix a maximal ideal $M$ of $R$ containing $p$.
Let $\mathcal{O}$ be the localization of $R$ at $M$, so that
$$\mathcal{O}=\{ a/b: a\in R,\; b\in R\smallsetminus M \}.$$
Then $\mathcal{O}$ is a local ring with a unique maximal ideal $\mathcal{M}=\{ a/b:a\in M,\; b\in R\smallsetminus M \}$,
and its residue field $k := \mathcal{O}/\mathcal{M}$ is an algebraic closure of the prime field $\mathbb{F}_p$ of $p$ elements.
The canonical quotient map $* \colon \mathcal{O} \to k = \mathcal{O}/\mathcal{M}$ extends naturally to an epimorphism $* \colon \mathcal{O}[G] \to k[G]$ of group algebras, providing the standard bridge between representations in characteristics 0 and $p$.

The group algebra $\mathcal{O}[G]$ decomposes as a direct sum of indecomposable two-sided ideals:
\[
\mathcal{O}[G] = B_1 \oplus B_2 \oplus \cdots \oplus B_t.
\]
These summands $B_i$ are called the \emph{$p$-blocks} of $G$.
Under the map $*$, each block $B_i$ maps onto an indecomposable two-sided ideal of $k[G]$, and this induces a bijection between the $p$-blocks of $\mathcal{O}[G]$ and those of $k[G]$. When no confusion arises, we identify each $p$-block with its image in $k[G]$.

Let $\mathrm{Irr}(G)$ and $\mathrm{IBr}(G)$ denote the sets of irreducible characters and irreducible $p$-Brauer characters of $G$, respectively.
Each $p$-block $B$ of $G$ uniquely determines the subsets $\mathrm{Irr}(B) \subseteq \mathrm{Irr}(G)$ and $\mathrm{IBr}(B) \subseteq \mathrm{IBr}(G)$ of ordinary and $p$-Brauer characters belonging to $B$.
It is well known that
the two sets are linked via the map $*$ described above.
Note that $\mathrm{Irr}(G)$ (resp. $\mathrm{IBr}(G)$) is the disjoint union of $\mathrm{Irr}(B)$ (resp. $\mathrm{IBr}(B)$) as $B$ runs over all $p$-blocks of $G$.
In particular, two $p$-blocks of $G$ coincide if and only if they 
share a common irreducible character (or irreducible $p$-Brauer character).
The unique $p$-block containing the principal character $1_G$ is called the \emph{principal $p$-block} of $G$, and is denoted by $B_0(G)$.

Let $B$ be a $p$-block of $G$, and let $x^G$ be a $G$-conjugacy class. We say that $x^G$ is a \emph{defect class} of $B$ if
\[
\frac{1}{|G|}\sum_{\chi\in \mathrm{Irr}(B)} \chi(1)\chi(x^{-1}) \notin \mathcal{M},\qquad \frac{|x^G|\psi(x)}{\psi(1)} \notin \mathcal{M}
\]
for some $\psi\in \mathrm{Irr}(B)$. The Sylow $p$-subgroups of $\mathrm{C}_{G}(x)$, where $x^G$ is a defect class of $B$, are called the \emph{defect groups} of $B$. In particular, the defect groups of the principal $p$-block $B_0(G)$ are the Sylow $p$-subgroups of $G$.

Now let $N$ be a normal subgroup of $G$.
Given a $p$-block $b$ of $N$, we say that a $p$-block $B$ of $G$ \emph{covers} $b$ if there exists a character $\chi \in \mathrm{Irr}(B)$ 
such that $\chi_N$ has an irreducible constituent in $\mathrm{Irr}(b)$ (see \cite[Theorem 9.2]{navarro98}).
Next, suppose that $b$ is a $p$-block of the quotient group $G/N$.
We say a $p$-block $B$ of $G$ \emph{dominates} $b$ provided $\mathrm{Irr}(b) \subseteq \mathrm{Irr}(B)$ (see \cite[Ch.5, Lemma 8.6]{nagao89}), where characters of $G/N$ are identified with their inflations to $G$.

We first present several basic facts on principal $p$-blocks that will be used frequently without mention.

\begin{lem}\label{lem: basics on principal blocks}
  Suppose that $N$ is a normal subgroup of a finite group $G$. 
  Then the following hold.  
  \begin{enumerate}[\rm (1)]
      \item $\mathrm{Irr}(B_0(G/N))\subseteq \mathrm{Irr}(B_0(G))$.
  \item For $\chi \in \mathrm{Irr}(B_0(G))$, every irreducible constituent of $\chi_N$ lies in $\mathrm{Irr}(B_0(N))$.
    \item For $\theta\in \mathrm{Irr}(B_0(N))$, there exists some $\chi \in \mathrm{Irr}(B_0(G))$ lying over $\theta$. 
  \item If $B_0(G)$ is the unique block of $G$ covering $B_0(N)$, then $\mathrm{Irr}(G|\theta) \subseteq \mathrm{Irr}(B_0(G))$ for every $\theta \in \mathrm{Irr}(B_0(N))$.  
    This happens, for instance, if $G/N$ is a $p$-group, or if $\mathrm{C}_G(P)\leq N$ for some $P\in \mathrm{Syl}_{p}(N)$.
  \end{enumerate}
\end{lem}
\begin{proof}
  Part (1) follows by the fact that $B_0(G)$ dominates $B_0(G/N)$, see, for instance, \cite[Pages 198, 199]{navarro98}. 
  Part (2) follows from \cite[Theorem 9.2]{navarro98}.
  Part (3) follows from \cite[Theorems 9.2, 9.4]{navarro98}.
  Part (4) follows from \cite[Theorem 9.4, Corollary 9.6, Theorem 6.7]{navarro98} and \cite[Lemma 3.1]{navarro12}.
\end{proof}

Now, we recall some facts on central product of finite groups.
Given distinct subgroups $H_i \leq G$ for $1 \leq i \leq t$,
suppose that $H_i$ and $H_j$ centralize each other whenever $i \neq  j$. 
If $G=\prod_{i=1}^{t} H_i$, we say that $G$ is
the \emph{central product} of its subgroups $H_i$,
and write $G=H_1\circ H_2\circ \cdots \circ H_t$.
Note that in this case, $H_i\unlhd G$
for all $i$ and $H_i \cap H_j \leq \mathrm{Z}(G)$ if $i \neq j$.

For $B$ a $p$-block of $G$, $Z\unlhd G$ and $\lambda\in \mathrm{Irr}(Z)$, we
set $\mathrm{Irr}(B|\lambda):=\mathrm{Irr}(B)\cap \mathrm{Irr}(G|\lambda)$.

\begin{lem}\label{lem: cent prod of chars and blocks}
  Let $G=H_1\circ H_2\circ \cdots \circ H_t$ be the central product of its subgroups $H_i$, and set $Z = \bigcap_{i=1}^{t} H_i$.
  Then $\mathrm{Irr}(G)$ is a disjoint union of the subsets $\mathrm{Irr}(G|\lambda)$
  as $\lambda$ runs through $\mathrm{Irr}(Z)$
  and, 
  for each $\lambda\in \mathrm{Irr}(Z)$, there exists a natural bijection
   \[
    \Lambda_\lambda:~\mathrm{Irr}(H_1|\lambda)\times \mathrm{Irr}(H_2|\lambda)\times \cdots \times\mathrm{Irr}(H_t|\lambda) \rightarrow \mathrm{Irr}(G|\lambda)
    \]
 satisfying 
\[
\Lambda_{\lambda}(\alpha_1,\alpha_2,\dots,\alpha_t)(h_1h_2\cdots h_t)=\alpha_1(h_1)\alpha_2(h_2)\cdots\alpha_t(h_t)
\]
for all $\alpha_i\in \mathrm{Irr}(H_i|\lambda)$ and $h_i\in H_i$.
  Suppose further that $Z$ is a $p$-group for a prime $p$.
   Then for every $\lambda\in\mathrm{Irr}(Z)$, the map $\Lambda_\lambda$ restricts to a natural bijection from
  $\mathrm{Irr}(B_0(H_1)|\lambda)\times \mathrm{Irr}(B_0(H_2)|\lambda) \times \cdots \times\mathrm{Irr}(B_0(H_t)|\lambda)$ 
  onto $\mathrm{Irr}(B_0(G)|\lambda)$,
  where $B_0(H_i)$ and $B_0(G)$ denote the principal $p$-blocks of the corresponding groups respectively.
  \end{lem}
\begin{proof}
  This is a combination of \cite[Theorem 10.7]{navarro18} and \cite[Lemma 4.1]{malle23}.
\end{proof}

 In the following, 
   we denote $\Lambda_{\lambda}(\alpha_1,\alpha_2,\dots,\alpha_t)$ in the above lemma 
   simply by $\alpha_1\circ \alpha_2\circ \cdots\circ \alpha_t$
   for every $\lambda\in \mathrm{Irr}(Z)$ and $\alpha_i\in \mathrm{Irr}(H_i|\lambda)$.

	We will also employ the Alperin--Dade's theory of isomorphic blocks.

	\begin{lem}\label{lem: isoblock} 
    Let $N$ be a normal subgroup of a finite group $G$, and let $p$ be a prime.
    Assume that $G/N$ has order coprime to $p$ and that $G=N \mathrm{C}_{G}(P)$ for some $P\in \mathrm{Syl}_{p}(N)$.
    Then restriction of characters defines a bijection between $\mathrm{Irr}(B_{0}(G))$ and $\mathrm{Irr}(B_{0}(N))$.
  \end{lem}
  \begin{proof}
    This is a well-known result: the case where $G/N$ is solvable was proved in \cite{alperin76}, and the general case was established in \cite{dade77}.
\end{proof}

  We end this subsection with Harris' classification of the finite groups with a unique $p$-block.

\begin{thm}[\mbox{\cite[Theorem 1]{harris85}}]\label{thm: ub}
  Given a prime $p$,
  a finite group $G$ has a unique $p$-block 
  if and only if one of the following holds.
  \begin{enumerate}[\rm (1)]
       \item $G$ is $p$-constrained with $\mathrm{O}_{p'}(G)=1$. 
    \item $p=2$, all components of $G$ are of type $M_{22}$ or $M_{24}$, and $\mathrm{O}_{2'}(G)=1$.
  \end{enumerate}
\end{thm}

\subsection{Finite groups of Lie type}\label{sec: 2.3}

Finally, we collect some standard facts about finite groups of Lie type, including the Deligne--Lusztig theory for characters
and the block decomposition of these groups.
For further details, 
we refer to \cite{digne91, cabanes04}.

By a finite simple group $S$ of Lie type in characteristic $\ell$, we mean a group of the form $S = [G, G]$, where $G := \mathbf{G}^F$ is the fixed-point group of
a simple algebraic group $\mathbf{G}$ of adjoint type over the algebraic closure $\overline{\mathbb{F}_\ell}$
under a Steinberg endomorphism $F$ of $\mathbf{G}$.
Let $(\mathbf{G}^*, F^*)$ be in duality with $(\mathbf{G}, F)$, and set $G^* := (\mathbf{G}^*)^{F^*}$.
Then (rational) \emph{Lusztig series} $\mathcal{E}(G, s)$ give a partition of the irreducible characters of $G$ as 
\begin{equation}\label{eq_1}
   \mathrm{Irr}(G)=\bigsqcup_{s} \mathcal{E}(G, s)  
\end{equation}
where $s$ runs over a set of representatives for the $G^*$-conjugacy classes of semisimple elements of $G^*$.
To each $G^*$-conjugacy class of a semisimple element $s \in G^*$, 
since $\mathrm{C}_{\mathbf{G}^*}(s)$ is connected,
there corresponds a unique irreducible character $\chi_s \in \mathcal{E}(G, s)$ of degree
\[
\chi_s(1) = \left[G^* : \mathrm{C}_{G^*}(s)\right]_{\ell'},
\]
the $\ell'$-part of the index $\left[G^* : \mathrm{C}_{G^*}(s)\right]$ (see \cite[\S 14]{digne91}).
Such a character $\chi_s$ is called the \emph{semisimple character} of $G$ associated with $s$.

We now turn to the block theory of finite groups of Lie type.
For every semisimple $p'$-element $s \in G^*$, we define the union of Lusztig series
\begin{equation}\label{eq_2}
  \mathcal{E}_p(G, s) := \bigcup_{t} \mathcal{E}(G, st),
\end{equation}
where $t$ ranges over all $p$-elements in $\mathrm{C}_{G^*}(s)$.
For semisimple $p'$-elements $s, s' \in G^*$, the sets $\mathcal{E}_p(G, s)$ and $\mathcal{E}_p(G, s')$ either coincide or are disjoint. 
 Precisely, $\mathcal{E}_p(G, s) = \mathcal{E}_p(G, s')$ if and only if $s$ and $s'$ are conjugate in $G^*$.
Since every semisimple element of $G^*$ 
decomposes as a product of a semisimple $p'$-element $s\in G^*$ and some $p$-element $t\in \mathrm{C}_{G^*}(s)$,
it follows by (\ref{eq_1}) and (\ref{eq_2}) that
\[
\mathrm{Irr}(G)=\bigsqcup_{s} \mathcal{E}_p(G,s),
\]
  where $s$ runs over a set of representatives for the $G^*$-conjugacy classes of semisimple $p'$-elements of $G^*$.
 A fundamental theorem of Brou\'e and Michel \cite[Theorem 2.2]{broue89} states that each set $\mathcal{E}_p(G, s)$ is a union of $\mathrm{Irr}(B)$ for some $p$-blocks $B$ of $G$. 
Together with a complementary result of Hiss in \cite{hiss90}, we obtain the following key theorem.
For readers' convenience, we present a proof here.

\begin{thm}\label{thm: blocks of lie type}
  Let $G$ and $G^*$ be as above,
  and let $B$ be a $p$-block of $G$.
  Then there exists a semisimple $p'$-element $s\in G^*$ (unique up to $G^*$-conjugacy)
  such that $\mathrm{Irr}(B)\subseteq \mathcal{E}_p(G,s)$ and $\mathrm{Irr}(B)\cap\mathcal{E}(G,s)\neq \emptyset$.
  In particular, $\mathrm{Irr}(B_0(G))\subseteq\mathcal{E}_p(G,1)$.
\end{thm}
\begin{proof}
  Note that $\mathrm{Irr}(G)$ are partitioned into $\mathcal{E}_p(G,s)$
  where $s$ runs over a set of representatives for the $G^*$-conjugacy classes of semisimple $p'$-elements of $G^*$.
  By \cite[Theorem 9.12]{cabanes04}, there exists a semisimple $p'$-element $s\in G^*$
  such that $\mathrm{Irr}(B)\subseteq \mathcal{E}_p(G,s)$ and $\mathrm{Irr}(B)\cap \mathcal{E}(G,s)\neq \varnothing$.
  Since the principal character $1_G\in \mathrm{Irr}(B_0)\cap\mathcal{E}(G,1)$,
  we are done by \cite[Definition 9.13]{cabanes04}.
\end{proof}

\section{Real blocks}

We begin this section by recalling the classification of the finite quasi-simple groups with a unique real $2$-block, due to McHugh and Schaeffer Fry. 
This classification was established in answer to  
 a question posed by Gow and Murray on the existence of non-trivial $2$-Brauer characters of quadratic type in non-abelian simple groups.

\begin{thm}[\mbox{\cite[Theorem 1.1]{mchugh26}}]\label{thm: urb}
  If $G$ is a finite quasi-simple group, then $G$ has a unique real $2$-block
  if and only if $G/\mathrm{Z}(G)\in \{ M_{11},M_{22},M_{23}, M_{24}, \mathrm{PSL}_{3}(3),\mathrm{PSU}_{3}(3) \}$. 
\end{thm}

Since distinct $p$-blocks of a finite group $G$ are disjoint with respect to both irreducible characters and irreducible $p$-Brauer characters, 
there is an alternative definition of real $p$-blocks as follows:
a $p$-block $B$ of $G$ is real if $\overline{\chi} \in \mathrm{Irr}(B)$ for some character $\chi \in \mathrm{Irr}(B)$.
  In particular, the principal $p$-block is always real.
For $2$-blocks, 
recall that this definition admits an equivalent and simpler characterization: a $2$-block is real if and only if it contains at least one real-valued irreducible character.

\begin{lem}\label{lem: basic}
	Let $p$ be a prime, and let $G$ be a finite group.
  Then the following hold.
  \begin{enumerate}[\rm (1)]
       \item Let $N\unlhd G$ and assume that either $p\nmid |N|$ or $N\leq \mathrm{Z}(G)$. 
       If $G$ has a unique real $p$-block, then $G/N$ has a unique real $p$-block.
		\item  If $N$ is a normal subgroup of $G$ of odd index in $G$,
    then
    $G$ has a unique real $2$-block if and only if $N$ has a unique real $2$-block.
    \item  Suppose that, using the notation of Lemma {\rm\ref{lem: cent prod of chars and blocks}}, $G$ is the central product of its subgroups $H_i$ where $1 \leq i\leq t$.
    If $Z:=\bigcap_{i=1}^{t}H_i$ is a $2$-group,
    then $G$ has a unique real $2$-block if and only if every $H_i$ has a unique real $2$-block.
  \end{enumerate}
\end{lem}
\begin{proof}
    (1) 
    Let $b$ be a $p$-block of $G/N$.
    By \cite[Page 199, Theorems 9.9, 9.10]{navarro98}, 
    there exists a unique $p$-block $B$ of $G$ dominating $b$,
    and $\mathrm{Irr}(b)=\mathrm{Irr}(B)\cap \mathrm{Irr}(G/N)$. 
    By the definition of real blocks, 
    if $b$ is real then so is $B$.
     Now suppose $b$ is a real $p$-block distinct from the principal $p$-block $b_0$ of $G/N$.
    Since $\mathrm{Irr}(b)\cap \mathrm{Irr}(b_0)=\varnothing$,
    and $B_0(G)$ dominates $b_0$ by Lemma \ref{lem: basics on principal blocks}(1), it follows that $B$ is a real $p$-block of $G$ distinct from $B_0(G)$. 
    This yields a contradiction.

  (2) We first assume that $G$ has a unique real $2$-block. 
  Let $\theta \in \mathrm{Irr}(N)$ be a real-valued character. 
  Since the index $[G:N]$ is odd, 
  by Lemma \ref{lem: real chars}(2), there exists a unique real-valued character $\chi \in \mathrm{Irr}(G|\theta)$. 
  By our hypothesis, $\chi$ lies in $B_0(G)$. 
  Consequently, Lemma \ref{lem: basics on principal blocks}(2) forces 
  $\theta \in \mathrm{Irr}(B_0(N))$. 
  This shows that every real-valued irreducible character of $N$ belongs to $B_0(N)$, so $N$ has a unique real $2$-block.
  

  We assume next that $N$ has a unique real $2$-block,
  and proceed by induction on $[G:N]$ to show that $G$ has a unique real $2$-block.
  If $[G:N]=1$, then we are done. 
  Hence, we may assume that $[G:N]>1$.
  Since $G/N$ has odd order, by Feit--Thompson's theorem and induction, 
  we may assume that $[G:N]=q$ for some odd prime $q$.
  Let $P\in \mathrm{Syl}_{2}(N)$, and observe that $P$ is also a Sylow $2$-subgroup of $G$.
  If $\mathrm{C}_{G}(P)\leq N$,
  as $B_0(G)$ is the unique $2$-block covering $B_0(N)$ by Lemma \ref{lem: basics on principal blocks}(4),
  it follows by Lemma \ref{lem: real chars}(1) that $G$ has a unique real $2$-block $B_0(G)$.
 
  So, we may assume that $G=N \mathrm{C}_{G}(P)$.
  By Lemma \ref{lem: isoblock}, restriction of characters defines a bijection $\mathrm{Irr}(B_0(G))\to\mathrm{Irr}(B_0(N))$. 
  Let $\chi\in \mathrm{Irr}(G)$ be a real-valued character, and let $\theta$ be an irreducible constituent of $\chi_N$.
  Then $\theta$ is also real-valued by Lemma \ref{lem: real chars}(1).
  As $N$ has a unique real $2$-block, it forces $\theta \in \mathrm{Irr}(B_0(N))$.
  Through the aforementioned bijection, there exists a unique $\alpha \in \mathrm{Irr}(B_0(G))$ such that $\alpha_N = \theta$. 
  Since $\mathrm{Irr}(B_0(G))$ is invariant under taking complex conjugation, $\overline{\alpha} \in \mathrm{Irr}(B_0(G))$, and we have $\overline{\alpha}_N = \overline{\theta} = \theta$. 
  The uniqueness of $\alpha$ then forces $\overline{\alpha} = \alpha$, meaning $\alpha$ is real-valued. 
  Furthermore, since there is a unique real-valued irreducible character of $G$ lying over $\theta$ by Lemma \ref{lem: real chars}(2), 
  we must have $\chi = \alpha$, which proves that $\chi \in \mathrm{Irr}(B_0(G))$. 
  Thus, $G$ has a unique real $2$-block.

  (3) Let $\chi\in \mathrm{Irr}(G)$.
   By Lemma \ref{lem: cent prod of chars and blocks},
   there exists some $\lambda\in \mathrm{Irr}(Z)$
   such that $\chi=\alpha_1\circ \cdots \circ \alpha_t$ for $\alpha_i\in \mathrm{Irr}(H_i|\lambda)$.
   Noting that $\chi$ is real-valued if and only if all $\alpha_i$ are real-valued,
   we conclude that part (3) holds by Lemma \ref{lem: cent prod of chars and blocks}.
\end{proof}

  We remark that for odd primes $p$, the converse of part (1) in the preceding lemma fails in general. 
  For example, take $G = \mathsf{D}_{2p} \times \mathsf{C}_2$ and set $N = \mathrm{Z}(G)$. 
  Then $G$ is solvable with $\mathrm{O}_{p'}(G)=N$, so that $G/N$ has a unique $p$-block, yet $G$ admits at least two distinct real $p$-blocks. 
  By contrast, the situation for $p=2$ is considerably better: there is a natural bijection between the real $2$-blocks of $G/N$ and those of $G$, whenever either $N\leq \mathrm{Z}(G)$,
  or $|N|$ is odd and every $G$-conjugacy class contained in $N$ has odd size.
  To establish this correspondence, we first prove the following auxiliary lemma.

Let $N$ be a normal subgroup of a finite group $G$. 
The Galois group $\mathcal{G}:=\mathrm{Gal}(\mathbb{Q}_{|G|}/\mathbb{Q})$ of \emph{the $|G|$-th cyclotomic extension} naturally acts on $\mathrm{Irr}(N)$ via
\[
\theta^{\sigma}(x) = \theta(x)^{\sigma}
\]
for all $\theta \in \mathrm{Irr}(N)$, $\sigma \in \mathcal{G}$ and $x \in N$. 
Denote by $\tau\in\mathcal{G}$ the complex conjugation.
Then 
$$\theta^{\tau}(x)=\overline{\theta(x)}=\theta(x^{-1})$$ 
for all $\theta \in \mathrm{Irr}(N)$ and $x\in N$.
Accordingly, $(x^{N})^{\tau}:=(x^{-1})^{N}$ defines an action of $\langle \tau\rangle$ on 
the set $\mathrm{Cl}(N)$ of $N$-conjugacy classes.
Under this action, a class $x^N\in\mathrm{Cl}(N)$ is $\tau$-invariant if and only if $x^N$ is a real $N$-conjugacy class.
Since the action of $\tau$ commutes with $G$-conjugation  
on both $\mathrm{Irr}(N)$ and $\mathrm{Cl}(N)$,
the direct product $\langle \tau\rangle \times G$ acts compatibly on these two sets.

\begin{lem}\label{lem: all real in G/N}
  Let $G$ be a finite group admitting an odd-order normal subgroup $N$.
  Assume that every $G$-conjugacy class contained in $N$ has odd size.
  Then every real-valued irreducible character of $G$ lies in $\mathrm{Irr}(G/N)$.
\end{lem}
\begin{proof}
  Let $\tau\in \mathrm{Gal}(\mathbb{Q}_{|G|}/\mathbb{Q})$ be the complex conjugation,
  and set $\Gamma=\langle \tau\rangle \times G$.
  Then both $\Gamma$ and $G$ act on $\mathrm{Irr}(N)$ and $\mathrm{Cl}(N)$.
  By \cite[Corollary 6.33]{isaacs76}, the number of orbits of $\Gamma$ on $\mathrm{Irr}(N)$ 
   equals the number of orbits of $\Gamma$ on $\mathrm{Cl}(N)$, and the same holds for $G$.

  Let $k$ be the number of $G$-orbits in $\mathrm{Irr}(N)$, let $m$ be the number of $\tau$-invariant $G$-orbits on $\mathrm{Irr}(N)$, and let $n$ be the number of $\tau$-invariant $G$-orbits on $\mathrm{Cl}(N)$.
  Then the total number of $\Gamma$-orbits on $\mathrm{Irr}(N)$ is $m+\frac{k-m}{2}=\frac{k+m}{2}$,
  while the total number of $\Gamma$-orbits on $\mathrm{Cl}(N)$ is $n+\frac{k-n}{2}=\frac{k+n}{2}$.
  Equating these two quantities yields $\frac{k+m}{2} = \frac{k+n}{2}$, so $m = n$.

We now show that $m=1$.
Since $m=n$, it suffices to prove $n=1$.
Let $x^G\subseteq N$ be a real $G$-conjugacy class.
Then there exists $g\in G$ such that $x^g=x^{-1}$.
Since any odd power of $g$ also inverts $x$, we may assume that $g$ is a $2$-element.
As $2\nmid|x^G|$, $\mathrm{C}_G(x)$ contains a Sylow $2$-subgroup of $G$.
Observe that
\[
\mathrm{C}_G(x)^g=\mathrm{C}_G(x^g)=\mathrm{C}_G(x^{-1})=\mathrm{C}_G(x),
\]
and so $\langle g\rangle$ normalizes some Sylow $2$-subgroup $P$ of $\mathrm{C}_{G}(x)$.
Hence $P\langle g\rangle$ is a $2$-subgroup of $G$, which forces $g\in P$, because $P$ is also a Sylow $2$-subgroup of $G$.
It follows that $x^2=1$.
Since $|N|$ is odd, we obtain $x=1$.
Consequently, there is exactly one $\tau$-invariant $G$-orbit on $\mathrm{Cl}(N)$, and thus $n=1$.

Therefore, for every non-trivial $\lambda\in \mathrm{Irr}(N)$,
  its $G$-orbit $\{ \lambda^g:g\in G \}$ is not $\tau$-invariant.
  By Clifford's theorem, this implies that every real-valued irreducible character of $G$ must lie in $\mathrm{Irr}(G/N)$.
\end{proof}

\begin{thm}\label{thm: bij real 2-blocks}
  Let $G$ be a finite group admitting a normal subgroup $N$.
  Suppose that one of the conditions holds:
  \begin{enumerate}[\rm (1)]
    \item $N\leq \mathrm{Z}(G)$;
    \item $|N|$ is odd and every $G$-conjugacy class contained in $N$ has odd size.
  \end{enumerate}
    Then block domination induces a bijection between the set of real $2$-blocks of $G/N$ and the set of real $2$-blocks of $G$.
\end{thm}
\begin{proof}
    {\bf Suppose that condition (2) holds}.
  Since every $G$-conjugacy class contained in the odd-order group $N$ has odd size, Lemma \ref{lem: all real in G/N} implies that all real-valued irreducible characters of $G$ belong to $\mathrm{Irr}(G/N)$. 
  Let $b$ be a real $2$-block of $G/N$, and let $B$ be the unique $2$-block of $G$ dominating $b$. 
  Then the assignment $b \mapsto B$ defines a map $f$ from the set of $2$-blocks of $G/N$
  to the set of $2$-blocks of $G$.
  As $N$ is of odd order, \cite[Theorem 9.9(c)]{navarro98} yields $\mathrm{Irr}(b)=\mathrm{Irr}(B)$. 
  In particular, $b$ is the unique $2$-block of $G/N$ dominated by $B$.
  Hence, the map $f$ is injective.

  Conversely, let $B$ be any real $2$-block of $G$. 
  By definition, there exists a real-valued character $\chi\in\mathrm{Irr}(B)$. 
  Note that $\chi\in\mathrm{Irr}(G/N)$, and so $\chi$ belongs to some real $2$-block $b$ of $G/N$. 
  It then follows that $B$ dominates $b$ (see \cite[Page 199]{navarro98}).
  Therefore, the map $f$ is surjective.

  Consequently, block domination defines a bijection between the set of real $2$-blocks of $G/N$ and the set of real $2$-blocks of $G$.

  {\bf Suppose that condition (1) holds}.
  We proceed by induction on $|N|$.
  Since $N$ is abelian, it decomposes as the direct product 
  of its Sylow $2$-subgroup and its Hall $2'$-subgroup. 
  By induction, we may assume that either $N$ is a $2$-group or a $2'$-group.

 If $N$ is a $2'$-group, then it has odd order. 
 Because $N \leq \mathrm{Z}(G)$, every $G$-conjugacy class contained in $N$ has size $1$, which is odd. Thus $N$ satisfies condition (2) and we are done.

  So, we may assume that $N$ is a 2-group.
 Let $b$ be a $2$-block of $G/N$, and let $B$ be the unique $2$-block of $G$ dominating $b$. 
  By \cite[Theorem 7.6]{navarro98}, the assignment $b \mapsto B$ defines a bijection between the $2$-blocks of $G/N$ and those of $G$, satisfying $\mathrm{IBr}(b) = \mathrm{IBr}(B)$.
Note that, by \cite[Theorems 3.33, 3.35]{navarro98}, a $2$-block is real if and only if it contains a real-valued irreducible $2$-Brauer character. It follows immediately that $b$ is real if and only if $B$ is real, which yields the desired bijection in this case.
\end{proof}

Let $G$ be a finite group.
Recall that $\chi\in \mathrm{Irr}(G)$ is said to have $p$-defect zero if $p\nmid |G|/\chi(1)$.
Notably, such a character never lies in the principal $p$-block of $G$ whenever $p\mid |G|$.
Recall further that the \emph{socle} of $G$ is defined as the product of all its minimal normal subgroups.
We say a finite group $A$ is \emph{almost simple} if $S \leq A \leq \mathrm{Aut}(S)$ for some finite non-abelian simple group $S$, in which case $S$ is the socle of $A$.

Let $\epsilon\in \{ \pm \}$.
We adopt unified notation for families of finite simple groups of Lie type:
we write $A_n^\epsilon(q)$, where $\epsilon=+$ corresponds to $A_n(q)\cong\mathrm{PSL}_{n+1}(q)$ and $\epsilon=-$ to ${}^{2}\!A_n(q)\cong\mathrm{PSU}_{n+1}(q)$;
analogous notation is used for the groups $D_{n}^{\epsilon}(q)$ and $E_{6}^{\epsilon}(q)$.
The proof of the following theorem largely adapts arguments from Marinelli and Tiep \cite[Section 4]{marinelli13}, combined with reasoning from Tiep \cite[Proposition 6.3]{tiep15}.

\begin{thm}\label{thm: almost simple}
   Let $A$ be an almost simple group with socle $S$.
   Assume that $S$ admits at least two distinct real $2$-blocks.
   Then 
    \[
  \begin{gathered}
     \text{$S$ has an irreducible character $\alpha\notin \operatorname{Irr}(B_0(S))$ such that there exist a subgroup $J$} \\
     \text{with $\mathrm{I}_{A}(\alpha)\leq J\leq A$ and a real-valued character $\beta\in \operatorname{Irr}(J)$ lying over $\alpha$.} 
   \end{gathered} \tag{$*$} 
   \]
\end{thm}
\begin{proof}
  Since $S$ admits at least two distinct real $2$-blocks, we may assume by Lemma \ref{lem: real chars}(2) that $2$ divides the index $[A:S]$.
  Moreover, Theorem \ref{thm: urb} excludes the cases $S\in \{ M_{22}, \mathrm{PSL}_3(3), \mathrm{PSU}_3(3) \}$. 
  Now, we proceed with the proof by establishing several claims.

  \noindent{\bf Claim 1}. If $S$ possesses an irreducible character $\alpha\notin \mathrm{Irr}(B_0(S))$
that extends to a real-valued character $\gamma$ of $\mathrm{I}_{\mathrm{Aut}(S)}(\alpha)$,
then $S$ has property ($*$).
In particular, finite simple groups of Lie type in even characteristic other than ${}^{2}\!F_4(2)'$ have property ($*$).

Let $T = \mathrm{I}_{A}(\alpha)$. 
For any subgroup $J$ such that $T\leq J\leq A$, Clifford's theorem and Lemma \ref{lem: res and ind} yield that the character $\beta:=(\gamma_T)^{J}$ is a real-valued irreducible character of $J$ lying over $\alpha$.
Now suppose $S$ is a finite simple group of Lie type in even characteristic other than ${}^{2}\!F_4(2)'$.
By the main result of \cite{feit93}, its Steinberg character $\alpha$, which lies outside $B_0(S)$, extends to a rational-valued (and hence real-valued) character of $\mathrm{Aut}(S)$. 
We are therefore done by the preceding discussion.

\noindent{\bf Claim 2}. Sporadic simple groups and groups $\mathsf{A}_5$, $\mathsf{A}_6,D_4(3)$ and ${}^{2}\!F_4(2)'$ have property ($*$).

 Since $2\mid [A:S]$, using $\mathsf{GAP}$ \cite{gap},
 we have either $S\in \{ M_{12}, J_2, Suz, HS, McL, He, Fi_{22}, Fi_{24}', HN, O'N, J_3 \}$ 
 or $S \in \{ \mathsf{A}_5, \mathsf{A}_6,D_4(3),{}^{2}\!F_4(2)' \} $.
 We then verify via $\mathsf{GAP}$ \cite{gap} that $S$ admits an irreducible character $\alpha\notin \mathrm{Irr}(B_0(S))$
 that extends to a real-valued character of $\mathrm{I}_{\mathrm{Aut}(S)}(\alpha)$, meaning we are done by Claim 1.
 Indeed, if 
 $$S\in \{ He, Fi_{22}, Fi_{24}', HN, O'N, J_3, \mathsf{A}_5, \mathsf{A}_6, D_4(3),{}^{2}\!F_4(2)' \},$$
  we may choose such a character $\alpha$ to have $2$-defect zero, with respective degrees 21504, 1441792, 197813862400, 3424256, 175616, 1920, 4, 8, 716800, or 2048;
 if $S\in \{ M_{12}, J_2, Suz, HS, McL \}$,
 we may choose such a character $\alpha$ with respective degrees $144$, $160$, $66560$, $1408$, or $3520$.

 \noindent{\bf Claim 3}. The groups $\mathsf{A}_n$, for $n>6$, have property ($*$).

Note that $|\mathrm{Aut}(S)/S| = 2$, and so $A = \mathrm{Aut}(S) \cong \mathsf{S}_n$.
Note also that every irreducible character of $A$ is real-valued.
By Theorem \ref{thm: urb}, there exists a real-valued character $\alpha \in \mathrm{Irr}(S) \smallsetminus \mathrm{Irr}(B_0(S))$.
Since any character $\beta\in \mathrm{Irr}(A|\alpha)$ is also real-valued,
 we may therefore take $J=A$ to complete the claim.



\noindent{\bf Claim 4}. Finite simple groups of Lie type in odd characteristic $\ell$ have property ($*$).

For the remainder of the proof, 
we may assume that $S$ is neither isomorphic to any of the simple groups of Lie type listed in Claim 2, nor to $\mathrm{PSL}_{3}(3)$ or $\mathrm{PSU}_{3}(3)$.
Following the setup from Subsection \ref{sec: 2.3},
we set $S = [G, G]$, where $G := \mathbf{G}^F$ for a simple algebraic group $\mathbf{G}$ of adjoint type over the algebraic closure $\overline{\mathbb{F}_\ell}$ and a Steinberg endomorphism $F$ of $\mathbf{G}$.
Let $(\mathbf{G}^{*}, F^{*})$ be in duality with $(\mathbf{G}, F)$, and set $G^*:=\mathbf{G}^{*F^*}$.

It suffices to find a character $\alpha\in \mathrm{Irr}(S)$
that is the restriction to $S$ of a semisimple character $\chi_s\in \mathrm{Irr}(G)$ corresponding to an odd-order non-trivial semisimple element $s\in G^*$, with $\operatorname{gcd}(o(s),|\mathrm{Z}(G^*)|)=1$,
such that there exist a subgroup $J$ with $\mathrm{I}_{A}(\alpha)\leq J\leq A$ and a real-valued character $\beta\in \mathrm{Irr}(J)$ lying over $\alpha$.
Indeed, since the non-trivial semisimple element $s$ has odd order and $\gcd(o(s),|\mathrm{Z}(G^*)|)=1$, $sz$ is not a $2$-element for any $z\in\mathrm{Z}(G^*)$; 
this implies $\chi_{sz}\notin\mathcal{E}_2(G,1)$, and Theorem~\ref{thm: blocks of lie type} therefore forces that no semisimple character $\chi_{sz}$ lies in $\mathrm{Irr}(B_0(G))$ for all $z\in\mathrm{Z}(G^*)$;
as shown in the proof of \cite[Proposition~4.3]{marinelli13}, we have 
$$\mathrm{Irr}(G|\alpha)=\{\chi_{sz}:z\in\mathrm{Z}(G^*)\};$$
Lemma~\ref{lem: basics on principal blocks}(3) then guarantees that $\alpha\notin\mathrm{Irr}(B_0(S))$.

We first assume that $S$ is not one of the following groups: $A_n^\epsilon(q) $ with $ n\geq 2$, $D_{2n+1}^\epsilon(q)$ with $n\geq 2$, or $E_6^{\epsilon}(q)$.
If $S\cong D_4(q)$, we adapt the argument from \cite[Proposition 4.9]{marinelli13}.
For all other groups in this family, we follow the proofs of \cite[Propositions 4.4,~4.5]{marinelli13}.

We next assume that $S$ is one of the following groups: $A_n^\epsilon(q) $ with $ n\geq 2$, $D_{2n+1}^\epsilon(q)$ with $n\geq 2$, or $E_6^{\epsilon}(q)$.
Let $\tau$ denote the outer automorphism of $S$ lifting the graph automorphism of $\mathbf{G}$ of order 2,
and observe that $\tau$ is the unique involution in $\mathrm{Aut}(S)/S$.
If $S\cong\mathrm{PSU}_n(q)$ with $2\mid n$ and $\tau\in GA/S$,
then we follow the proof of \cite[Proposition 6.3]{tiep15}.
If $S\cong \mathrm{PSL}_3(q)$ where $q=\ell^f\neq 3$ with $2\nmid f$, or $S\cong \mathrm{PSU}_3(q)$ with $q\neq 3$,
noting that $\tau\in GA/S$ because $2\mid [A:S]$,
we then use Cases IIa and IIb from the proof of \cite[Proposition 4.7]{marinelli13}.
For the remaining groups in this family, 
we follow the argument of \cite[Proposition 4.7]{marinelli13}.
\end{proof}

\section{Proof of the main results}

We first establish the following theorem, which enables us to analyze the components of finite groups possessing a unique real $2$-block.

\begin{thm}\label{thm: reduction}
   Let $G$ be a finite group with a unique real $2$-block.
   If $L$ is a component of $G$,
   then $L$ also has a unique real $2$-block.
\end{thm}
\begin{proof}
  Let $G$ be a counterexample of minimal order.
  Then $G$ admits a component $L$
  with at least two distinct real $2$-blocks.
  By Theorem \ref{thm: bij real 2-blocks}, $L/\mathrm{Z}(L)$ also has at least two distinct real $2$-blocks.
  Note that $G/\mathrm{O}_{2'}(G)$ has a unique real $2$-block by Lemma \ref{lem: basic}(1),
  and that $L\mathrm{O}_{2'}(G)/\mathrm{O}_{2'}(G)$ is a component of $G/\mathrm{O}_{2'}(G)$.
  The minimality of $G$ therefore forces $\mathrm{O}_{2'}(G) = 1$.
  
  Let $N=\langle L^{x}:x\in G\rangle$ be the normal closure of $L$ in $G$, set $Z = \mathrm{Z}(N)$, and write $\overline{G} = G/Z$.
  Then $Z$ is a $2$-group because $\mathrm{O}_{2'}(G)=1$.
  Moreover, by \cite[6.5.3, 1.6.7, 1.6.3(b)]{kurzweil04},
  \[
\overline{N} = \overline{L_1} \times \overline{L_2} \times \cdots \times \overline{L_t}
\]
  is a minimal normal subgroup of $\overline{G}$,
  where $\overline{L_i}=\overline{L^{x_i}}$ are non-abelian simple groups with suitable $x_i\in G$.
  Without loss of generality, we may assume that $x_1=1$.
  Since $Z$ is a central $2$-subgroup of $N$,
  \cite[Theorem 9.10]{navarro98}
  implies that
  $$\mathrm{Irr}(B_0(N))\cap \mathrm{Irr}(\overline{N})=\mathrm{Irr}(B_0(\overline{N})).$$
  Now set $\overline{H}=\mathrm{N}_{\overline{G}}(\overline{L})$ and $\overline{C}=\mathrm{C}_{\overline{G}}(\overline{L})$.
  Then $\overline{H}/\overline{C}$ is an almost simple group with socle $\overline{LC}/\overline{C} \cong L/\mathrm{Z}(L)$, which we have already noted has at least two distinct real $2$-blocks.
  Applying Theorem \ref{thm: almost simple} to the almost simple group $\overline{H}/\overline{C}$ and its socle $\overline{LC}/\overline{C}$,
  we deduce that $\overline{LC}=\overline{L}\times \overline{C}$ has an irreducible character $\alpha\times 1_{\overline{C}}\notin \mathrm{Irr}(B_0(\overline{LC}))$
  such that 
  there exist a subgroup $\overline{J}$ with $\mathrm{I}_{\overline{H}}(\alpha\times 1_{\overline{C}})\leq \overline{J}\leq \overline{H}$ 
  and a real-valued character $\beta\in \mathrm{Irr}(\overline{J})$ lying over $\alpha\times 1_{\overline{C}}$.
  In particular, $\alpha$ does not lie in $B_0(\overline{L})$.
  Let $\theta=(\alpha\times 1_{\overline{C}})_{\overline{N}}$.
  Then $\theta$ is an irreducible character of $\overline{N}$
  such that
  $\mathrm{I}_{\overline{G}}(\theta)=\mathrm{I}_{\overline{H}}(\alpha)=\mathrm{I}_{\overline{H}}(\alpha\times 1_{\overline{C}})$.
  By Clifford's theorem, the induced character $\chi:=\beta^{\overline{G}}$ is irreducible.
  Since $\beta$ is real-valued, $\chi$ is also real-valued.
  Note that $G$ has a unique real $2$-block, and so $\chi\in \mathrm{Irr}(B_0(G))$.
  As $\theta$ is an irreducible constituent of $\chi_N$, $\theta$ must lie in $\mathrm{Irr}(B_0(N))\cap \mathrm{Irr}(\overline{N})=\mathrm{Irr}(B_0(\overline{N}))$. 
  However, $\alpha$ does not lie in $B_0(\overline{L})$, so $\theta$ cannot lie in $B_0(\overline{N})$ by Lemma \ref{lem: cent prod of chars and blocks}, a contradiction.
\end{proof}

Note that a finite group $G$ with a unique $p$-block 
must satisfy $\mathrm{O}_{p'}(G)=1$.
However, this is not the case for a finite group with a unique real $2$-block.
For instance, 
the finite quasi-simple group
$G=3.M_{22}$ has a unique real $2$-block,
but $\mathrm{O}_{2'}(G)>1$.
So, in the next two results,
we deal with the properties of $\mathrm{O}_{2'}(G)$
for finite groups $G$ with a unique real $2$-block.

\begin{lem}\label{lem: conj class in N}
Let $G$ be a finite group admitting a normal subgroup $N$ of odd order.
 Assume that $G$ has exactly one real $2$-block. 
Then all $G$-conjugacy classes contained in $N$ have odd size.
\end{lem}
\begin{proof}
   Let $x$ be an element of $N$, and let $P\in \mathrm{Syl}_{2}(\mathrm{C}_{G}(x))$.
   Then $x^G$ is a defect class for some real 2-block $B$ of $G$ by \cite[Lemma 5.8]{gow00}.
   Since $G$ has a unique real $2$-block, $B=B_0(G)$ and $P$ is a Sylow $2$-subgroup of $G$.
   So, $|x^{G}|=[G:\mathrm{C}_{G}(x)]$ is odd.
\end{proof}

\begin{lem}\label{lem: G/N unique real block, conjugacy size odd}
  If $G$ is a finite group,
  then $G$ has a unique real $2$-block 
  if and only if $G/\mathrm{O}_{2'}(G)$ has a unique real $2$-block,
  and every $G$-conjugacy class in $\mathrm{O}_{2'}(G)$ has odd size.
\end{lem}
\begin{proof}
If $G$ admits only one real $2$-block, then the claim follows from Lemmas \ref{lem: basic}(1) and \ref{lem: conj class in N}.

Suppose $G/\mathrm{O}_{2'}(G)$ has a unique real $2$-block and all $G$-conjugacy classes in $\mathrm{O}_{2'}(G)$ are of odd size.
    Theorem \ref{thm: bij real 2-blocks} then implies $G$ also has exactly one real $2$-block.
\end{proof}

Before we establish Theorem \ref{thmA}, we need the following key theorem.

\begin{thm}\label{thm: M11 M23}
  Let $G$ be a finite group satisfying $G=\mathrm{O}^{2'}(G)$, and let 
  $N = S_1 \times S_2 \times \cdots \times S_t$ be a minimal normal subgroup of $G$, 
  where $S_i = S^{x_i}$ for some $x_i \in G$ with $x_1 = 1$. 
  Assume that $G$ has a unique real $2$-block.
  If $S\in \{ M_{11},M_{23}, \mathrm{PSL}_{3}(3),\mathrm{PSU}_{3}(3) \}$,
  then the following hold.
  \begin{enumerate}[\rm (1)]
    \item $\mathrm{N}_{G}(S)=S\mathrm{C}_{G}(S)$. 
    \item $G=N\times \mathrm{C}_{G}(N)$.
    \item Let $E$ denote the subgroup of $G$ generated by all components of $G$ of type $M_{11}$, $M_{23}$, $\mathrm{PSL}_{3}(3)$ or $\mathrm{PSU}_{3}(3)$.
    Then $G=E\times \mathrm{C}_{G}(E)$, where $E$ decomposes as a direct product of simple groups, each isomorphic to one of $M_{11}$, $M_{23}$, $\mathrm{PSL}_3(3)$ or $\mathrm{PSU}_3(3)$.
  \end{enumerate}
\end{thm}
\begin{proof} 
  (1) Set $H=\mathrm{N}_{G}(S)$ and $C=\mathrm{C}_{G}(S)$. 
  Since $|\mathrm{Out}(S)|=1$ for $S\in \{ M_{11}, M_{23} \}$,
  the equality $H = SC$ holds trivially in these cases.
  We may therefore assume that $S \in \{ \mathrm{PSL}_3(3), \mathrm{PSU}_3(3) \}$.
  
  Suppose, for contradiction, that $H > SC$. As $|\mathrm{Out}(S)| = 2$, we must have $[H : SC] = 2$.    
   Let $\alpha\in \mathrm{Irr}(S)$ be a character of $2$-defect zero ($\alpha(1)=2^4$ if $S\cong \mathrm{PSL}_{3}(3)$, $\alpha(1)=2^5$ if $S\cong \mathrm{PSU}_{3}(3)$), and set
   $\theta=\alpha\times 1_{S_2}\times \cdots \times 1_{S_t}\in \mathrm{Irr}(N)$.
   Then $\theta\notin \mathrm{Irr}(B_0(N))$, and its inertia group satisfies 
   $\mathrm{I}_{G}(\theta)=\mathrm{I}_{H}(\alpha)=\mathrm{I}_{H}(\alpha \times 1_{C})$.
   Note that $H/C$ is an almost simple group with socle $SC/C\cong S$.
   Checking via $\mathsf{GAP}$ \cite{gap}, one verifies that there exists a real-valued character $\psi \in \mathrm{Irr}(H)$ of $2$-defect zero lying over $\alpha \times 1_C$, and hence also lying over $\theta$.
   By Clifford correspondence, the induced character $\chi := \psi^G \in \mathrm{Irr}(G)$. 
   Moreover, since $\psi$ is real-valued, $\chi$ is also real-valued.
   However, $\theta$ does not lie in $B_0(N)$, so $\chi$ cannot belong to the principal $2$-block $B_0(G)$,
   which contradicts that $G$ has a unique real $2$-block.

   (2) 
   Since $\mathrm{N}_{G}(S_i)=S_i \mathrm{C}_{G}(S_i)$ for every $i$ by part (1), and 
   $\bigcap_{i=1}^{t} S_i\mathrm{C}_{G}(S_i)=N \mathrm{C}_{G}(N)$,
   we immediately have 
   $$\bigcap_{i=1}^{t} \mathrm{N}_{G}(S_i)=N \mathrm{C}_{G}(N).$$
    Set $B=N\mathrm{C}_{G}(N)$.
    Then the quotient group $G/B$ acts faithfully and transitively on the set $\{ S_1,S_2,\dots ,S_t \}$.
    Let $\Omega=\{ 1,2,\cdots ,t \}$.
     For each $g \in G$ and each $i \in \Omega$, we may write
    $x_i g = h_i(g) x_{i\pi(g)}$ with $h_i(g) \in H$. 
    This defines a homomorphism $\pi \colon G \to \mathsf{Sym}(\Omega)$ with kernel $B$.
    In particular, $G/B\cong \pi(G)$.

We claim that $\pi(G)$ contains no involutions. 
Suppose, for contradiction, that $\pi(g)$ is an involution for some $g \in G$, with cycle decomposition
$(i_1\,j_1)(i_2\,j_2)\cdots(i_s\,j_s)$.
By inspecting the character table of $S$ via $\mathsf{GAP}$ \cite{gap},
we fix a $2$-defect zero character $\alpha \in \mathrm{Irr}(S)$ that is not real-valued, 
and a non-principal real-valued $\beta\in \mathrm{Irr}(S)$.
Let $\Delta=\Omega\smallsetminus\{ i_1,j_1,\dots ,i_s,j_s \}$ denote the set of indices fixed by $\pi(g)$,
and set
  $$\theta= \alpha^{x_{i_1}}\times \overline{\alpha}^{x_{j_1}}\times \cdots \times \alpha^{x_{i_s}}\times \overline{\alpha}^{x_{j_s}}\times \bigtimes_{k\in \Delta} \beta^{x_k}.$$
  Since $\alpha$ is not real-valued, $\theta \in \mathrm{Irr}(N)$ is not real-valued.
  Let $T=\mathrm{I}_{G}(\theta)$, and observe that $N$ decomposes as a direct product $N = N_1 \times N_2 \times \cdots \times N_m$, where each $N_i$ is a minimal normal subgroup of $T$.
  Correspondingly, $\theta$ factors as $\theta = \theta_1 \times \theta_2 \times \cdots \times \theta_m$, with each $\theta_i \in \mathrm{Irr}(N_i)$ invariant under $T$.
  Recall that $\mathrm{N}_{G}(S_j)=S_j\mathrm{C}_{G}(S_j)$ for all $j\in \Omega$.
  Applying Lemma \ref{lem: tensor induction} to $(T,N_i,\theta_i)$, 
  one has that
  every $\theta_i$ extends to the tensor induced character $\chi_i \in \mathrm{Irr}(T)$ satisfying 
  $\mathrm{C}_T(N_i) \leq \ker(\chi_i)$.
Set $\psi = \prod_{i=1}^{m} \chi_i$. Then
\[
\psi_N = \prod_{i=1}^{m} (\chi_i)_{N_i}=\theta,
\]
so $\psi \in \mathrm{Irr}(T)$. 
  We now show that $\psi^g=\overline{\psi}$. 
  Observe that
  $(\psi^g)_N=(\psi_N)^{g}=\theta^g=\overline{\theta}$, and that
  \[
   T^g=\mathrm{I}_{G}(\theta^g)=\mathrm{I}_{G}(\overline{\theta})=\mathrm{I}_{G}(\theta)=T.
  \]
  For each $1\leq i\leq m$, since the perfect group $N_i^g$ is minimal normal in $T$,
  and given that 
  $$N_i^g=(N_i^g\cap N_1)\times (N_i^g\cap N_2) \times  \cdots \times (N_i^g\cap N_m)$$ 
  by \cite[1.6.3(a)]{kurzweil04},
  $g$ permutes the minimal normal subgroups $N_1,N_2,\dots,N_m$ of $T$.
  For each $N_i$, the character $\overline{\theta_i}$ extends uniquely to $\overline{\chi_i}$ via the tensor induction construction of Lemma \ref{lem: tensor induction}. It follows that $\overline{\psi} = \prod_{i=1}^{m} \overline{\chi_i}$ is the unique extension of $\overline{\theta}$ to $T$ obtained by this procedure. 
 Since $\theta_i^g\in \mathrm{Irr}(N_i^{g})$ also extends uniquely to $\chi_i^{g}$ via the tensor induction construction of Lemma 2.2,
  it follows that $\psi^g=\prod_{i=1}^{m} \chi_i^g\in \mathrm{Irr}(T)$
  is also an extension of $\theta^g=\overline{\theta}$ of the same form.
  So, we deduce that $\psi^g = \overline{\psi}$.
Now set $\chi = \psi^G$. By Clifford correspondence, $\chi \in \mathrm{Irr}(G)$. Moreover, we have
\[
\overline{\chi} = \overline{\psi^G} = (\overline{\psi})^G = (\psi^g)^G = \psi^G = \chi,
\]
where the fourth equality holds because character induction is invariant under $G$-conjugation. Hence $\chi$ is a real-valued irreducible character of $G$.
Since $G$ has a unique real $2$-block, which is necessarily the principal $2$-block $B_0(G)$, we have $\chi \in \mathrm{Irr}(B_0(G))$. As an irreducible constituent of $\chi_N$, the character $\theta$ must then lie in $B_0(N)$. 
However, $\alpha^{x_{i_1}}\in \mathrm{Irr}(S_{i_1})$ is a $2$-defect zero character which does not lie in 
$B_0(S_{i_1})$, which contradicts Lemma \ref{lem: cent prod of chars and blocks}.

Therefore, $\pi(G)$ contains no involutions, and thus the quotient $G/B\cong \pi(G)$ has odd order. Since $G = \mathrm{O}^{2'}(G)$ admits no non-trivial quotient of odd order, it follows that $G = B =N\mathrm{C}_G(N) = N \times \mathrm{C}_G(N)$.

  (3) Note that the Schur multiplier of $S$ is trivial whenever $S\in \{ M_{11},M_{23}, \mathrm{PSL}_{3}(3),\mathrm{PSU}_{3}(3) \}$.
It follows that $E \unlhd G$ decomposes as a direct product of simple groups, each isomorphic to one of $M_{11}$, $M_{23}$, $\mathrm{PSL}_3(3)$ or $\mathrm{PSU}_3(3)$.
  So, we write $E = E_1 \times E_2 \times \cdots \times E_t$, where each $E_i$ is a minimal normal subgroup of $G$.
  By part (2), we have $G=E_i\times \mathrm{C}_{G}(E_i)$ for every $1\leq i\leq t$.
  Taking intersections over all indices $i$, we conclude that
  $$G=\bigcap_{i=1}^{t} E_i \mathrm{C}_{G}(E_i)=E \times \mathrm{C}_{G}(E).$$ 
\end{proof}

Let $\mathrm{F}^*(G)$ denote the generalized Fitting subgroup of a finite group $G$.
Recall that $\mathrm{F}^*(G)=\mathrm{F}(G)\circ \mathrm{E}(G)$ where $\mathrm{F}(G)$ is the Fitting subgroup of $G$,
and $\mathrm{E}(G)$ is the \emph{layer} of $G$,
defined as the subgroup generated by all components of $G$.
Moreover, a fundamental property of the generalized Fitting subgroup of $G$
is 
$$\mathrm{C}_{G}(\mathrm{F}^*(G))\leq \mathrm{F}^*(G).$$

Now, we are ready to prove Theorem \ref{thmA}.

\begin{proof}[Proof of Theorem \ref{thmA}]
  By Lemmas \ref{lem: basic} and \ref{lem: G/N unique real block, conjugacy size odd}, $G$ admits a unique real $2$-block if and only if $T:=N/\mathrm{O}_{2'}(N)$ has a unique real $2$-block,
  and every $N$-conjugacy class in $\mathrm{O}_{2'}(N)$ has odd size.
  It therefore suffices to prove that $T$ has a unique real $2$-block 
  precisely when part (2) holds.
  Let $E$ denote the subgroup of $T$ generated by all the components of $T$ of type $M_{11}$, $M_{23}$, $\mathrm{PSL}_{3}(3)$ or $\mathrm{PSU}_{3}(3)$, and set $H=\mathrm{C}_{T}(E)$.
  By the construction of $E$, every component of $H$ is not of these four types.

  We first suppose that $T$ has a unique real $2$-block. 
  Then either $E=1$, or Theorem \ref{thm: M11 M23} implies $T=E\times H$, where $E$ is a direct product of simple groups
  each isomorphic to one of $M_{11}$, $M_{23}$, $\mathrm{PSL}_3(3)$, or $\mathrm{PSU}_3(3)$. 
  By Lemma \ref{lem: basic}(3), $H$ also has a unique real $2$-block. 
  Applying Theorems \ref{thm: reduction} and \ref{thm: urb} to $H$, we deduce that all components of $H$ are 
  of type $M_{22}$ or $M_{24}$.
  Furthermore, if $\mathrm{E}(H)=1$, as $\mathrm{O}_{2'}(H)\leq \mathrm{O}_{2'}(T)=1$, we have $\mathrm{F}^*(H)=\mathrm{O}_2(H)$ and
  $$\mathrm{C}_{H}(\mathrm{O}_{2}(H))=\mathrm{C}_{H}(\mathrm{F}^*(H))\leq \mathrm{F}^*(H)=\mathrm{O}_{2}(H)$$
  i.e. $H$ is $2$-constrained.

  Conversely, assume $T=E\times H$, where $E$ and $H$ satisfy conditions (2)(i) and (2)(ii). 
  By Theorem \ref{thm: urb} and Lemma \ref{lem: basic}(3), $E$ has a unique real $2$-block. 
  Since $\mathrm{O}_{2'}(H)=1$, Theorem \ref{thm: ub} guarantees that $H$ also admits a unique real $2$-block.
  A further application of Lemma \ref{lem: basic}(3) yields that $T$ has a unique real $2$-block, which completes the argument.
\end{proof}




\bigskip

\begin{acknowledgement}
No AI tool contributed to the mathematical reasoning, proof construction, or content of this work. The
large-language model Doubao (Seed-2.1) was used solely for English language polishing.
Gap 4.13.1, available at \url{http://www.gap-system.org}, was
used for all computations in support of this project.

   The authors gratefully acknowledge the support of the NSF of China (No. 12671032, 12501029).
  The authors are also grateful to the referee for her/his valuable comments.
\end{acknowledgement}

\end{document}